\documentclass[letterpaper, 10 pt, conference]{ieeeconf}  

\IEEEoverridecommandlockouts                              

\usepackage[utf8]{inputenc}
\usepackage{graphicx,xcolor}      
\usepackage{amsmath,amssymb,amsopn}
\usepackage{dsfont}
\usepackage{float}
\usepackage[caption = false]{subfig}
\usepackage{xparse}
\usepackage{hyperref}

\DeclareMathOperator*{\Argmin}{Argmin}

\newcommand{\Parens}[1]{\left( #1 \right)}
\newcommand{\Bracket}[1]{\left[ #1 \right]}
\newcommand{\Set}[1]{\left\{ #1 \right\}}

\newcommand{\Nmeas}{N_{\mathrm{mea}}}
\newcommand{\Nphy}{N_{\mathrm{phy}}}

\newcommand{\R}{\mathbb{R}}
\newcommand{\N}{\mathbb{N}}

\renewcommand{\L}{\mathfrak{L}}
\renewcommand{\H}{\mathcal{H}}
\renewcommand{\N}{\mathcal{N}}

\newtheorem{de}{Definition}
\newtheorem{theo}{Theorem}
\newtheorem{proposition}{Proposition}

\newtheorem{remarkTheo}{Remark}

\newenvironment{remark}[1][]{\begin{remarkTheo}}{\hfill$\square$\end{remarkTheo}}

\title{\LARGE \bf Physics-informed Learning for Identification and State Reconstruction of Traffic Density} 

\author{Matthieu Barreau$^{1}$, Miguel Aguiar$^{1}$, John Liu$^{1}$ and Karl Henrik Johansson${}^1$
	\thanks{$^{1}$ Division of Decision and Control Systems and Digital Futures, KTH Royal Institute of Technology, Stockholm, Sweden (e-mail: {\tt \{barreau,aguiar,johnliu,kallej\}@kth.se}).}%
	\thanks{This research is partially funded by the KAUST Office of Sponsored Research under Award No. OSR-2019-CRG8-4033, the Swedish Foundation for Strategic Research and Knut and Alice Wallenberg Foundation. 
	}
}

\begin{document}

\maketitle
\thispagestyle{empty}
\pagestyle{empty}

\begin{abstract}
    We consider the problem of traffic density reconstruction using measurements from probe vehicles (PVs) with a low penetration rate. In other words, the number of sensors is small compared to the number of vehicles on the road. The model used assumes noisy measurements and a partially unknown first-order model. All these considerations make the use of machine learning to reconstruct the state the only applicable solution. We first investigate how the identification and reconstruction processes can be merged and how a sparse dataset can still enable a good identification. Secondly, we propose a pre-training procedure that aids the hyperparameter tuning, preventing the gradient descent algorithm from getting stuck at saddle points. Examples using numerical simulations and the SUMO traffic simulator show that the reconstructions are close to the real density in all cases.
\end{abstract}

\section{Introduction}

Traffic state control has recently attracted a lot of attention~\cite{ferrara2018}. The possibility, in a near future, of using automated vehicles within the flow of vehicles opened many new control and observation perspectives. 

There are two classical methods for traffic state reconstruction. The most used one relies on a model, and is therefore labeled model-based. The survey~\cite{seo2017traffic} gives a good overview of many different modern techniques. The use of probe vehicles in this context is however quite recent and the reader can refer to~\cite{delle2019traffic,barreau2020dynamic,cicic2020numerical}. Nevertheless, for more complex models or robustness issues, it is today almost impossible to use macroscopic models.
This is in part due to the fact that it is very difficult, mathematically speaking, to derive convergence properties for infinite-dimensional nonlinear systems.

Another approach that is widely used today is called data-driven~\cite{berberich2020robust}. Such a methodology uses measurement data to derive system properties or predict the near future. This has been used in traffic state reconstruction in~\cite{herrera2010incorporation} for instance. The approach is quite powerful since it does not require many assumptions and the generality is quite high. However, there are not many practical applications since it requires many in-domain measurements. Using vehicles as sensors (probe vehicles) requires a large penetration rate, meaning that most of the vehicles are capable of probing.

To get the advantages of both methods while avoiding the aforementioned difficulties, it is of interest to develop a data-model driven methodology. This has recently been done in~\cite{raissi2019physics, sirignano2018dgm} using the notion of physics-informed learning. This data-based technique enforces a physical model on the measurements such that the generalization error is kept small even when the measurements are few and sparse. 

This approach has had a very large impact on the scientific community. The problem of traffic state reconstruction has recently been investigated using this methodology in~\cite{huang2020physics} and further in~\cite{shi2021physics,thodi2021incorporating}. The use of probe vehicles in particular is studied in~\cite{barreau2020learningbased}. The application to traffic flow reconstruction with external simulators is considered in~\cite{liu2020learning}.

The objective of this paper requires some definitions before we are able to state it precisely. However, informally speaking, the aim is to derive an algorithm capable of estimating the density of cars from measurements taken by a subset of the cars on the road.
The main contribution is the development of a machine learning framework for joint identification and state estimation of traffic flow from sparse measurements taken by probe vehicles.
In comparison with our previous papers~\cite{barreau2020learningbased,liu2020learning}, which focused on state estimation, here we consider in addition the problem of identifying the velocity function.
Such a difference implies an explosion of the number of variables and physical costs. The second contribution lies in the enhanced training procedure so that it can deal with many constraints simultaneously.

The paper is organized as follows. In Section~2, we introduce two traffic models. Section~3 is dedicated to the mathematical formulation of the objective in terms of an optimization problem and introduces a relaxed version of the problem. Section~4 proposes a learning solution focusing on the training procedure. Section~5 discusses the results obtained in different settings. Finally, Section~6 concludes with some perspectives on future research directions. 

\textbf{Notation:} We define $L_{\mathrm{loc}}^1(\mathcal{S}_1, \mathcal{S}_2)$, $L^{\infty}(\mathcal{S}_1, \mathcal{S}_2)$ and $C^k(\mathcal{S}_1, \mathcal{S}_2)$ as the spaces of locally integrable functions, bounded functions and continuous functions of class $k$ from $\mathcal{S}_1$ to $\mathcal{S}_2$, respectively. 
For a differentiable function $f$ of a single variable, $f'$ refers to its derivative, while partial derivatives of multivariable functions are indicated by subscripts.

\section{Coupled Macro-Micro Model of Traffic Flow}

There are two main kinds of traffic models in the literature, namely macroscopic and microscopic models~\cite{ferrara2018}. These two model types are briefly discussed in this first section to explain the benefit of using a coupled micro-macro model of traffic flow.

\subsection{Microscopic model}

Assume here that there are $N > 1$ vehicles located at position $y_i \in \R$ for $i \in \{ 1, \dots, N\}$. First order microscopic models assume that the following dynamic equation holds:
\[
    \left\{ \begin{array}{ll}
        y_i'(t) = V{\left( y_{i+1}(t) - y_i(t) \right)} & \text{ if } i \in \{1, \dots, N-1\}, \\
        y_N'(t) = V_{\mathrm{lead}}(t), \\
        y_1(0) < \dots < y_N(0).
    \end{array}
    \right.
\]
This is a follow-the-leader dynamical system when ${V_{\mathrm{lead}} \geq 0}$.
Usually, the velocity function $V$ depends on the intra-vehicular space and is decreasing, bounded and positive.

Such models perform badly in practice since they are too simple.
Second-order models have been introduced to improve the correlation with observations~\cite{ferrara2018}.
A drawback is that the system dimension becomes very large when dealing with many vehicles, so that an infinite-dimensional model (referred to as a macroscopic model) is introduced.

\subsection{Macroscopic model}

\subsubsection{Notion of density}

The \textit{normalized density of vehicles} $\rho: \R_{\geq 0} \times \R \to [0, 1]$ is defined such that
\[
    N_{\mathrm{veh}}(x_1, x_2, t) \propto \int_{x_1}^{x_2} \rho(t, x) dx,
\]
where $N_{\mathrm{veh}}(x_1, x_2, t)$ is the number of vehicles at a given time $t$ on the road segment $[x_1, x_2]$.
As shown in \cite{barreau2020learningbased,lighthill1955kinematic}, the normalized density is a solution of the equation
\begin{equation} \label{eq:hclaw}
    \rho_t(t, x) + f{\Parens{\rho(t, x)}}_x = 0
\end{equation}
with an initial condition $\rho(0, \cdot) = \rho_0$, and $f: [0, 1] \to \R_{\geq 0}$ is a smooth concave function which is related to $V$.

Equation~\eqref{eq:hclaw} corresponds to the law of conservation of mass of some substance with density $\rho$ and $f(\rho)$
is the rate at which the substance is passing through a point with density $\rho$.
Then we can also define the substance velocity ${v: [0, 1] \to \R_{\geq 0}}$ through the relation 
\begin{equation} \label{eq:rhoV}
    \rho v(\rho) = f(\rho).
\end{equation}
In traffic applications, the `substance' in question refers to the vehicles on a road, as from a macroscopic point of view one assumes that the flow of vehicles can be approximated by the flow of a continuous substance.
Then $f(\rho)$ gives the amount of vehicles flowing through some point of space per unit time, and $v(\rho)$ is the mean vehicle velocity.


\subsubsection{Existence and uniqueness}

It is well known that there may be no smooth solution of~\eqref{eq:hclaw}, even if $\rho_0$ is smooth.
Hence we consider solutions in a weak sense, and one then has the following existence and stability result~\cite{barreau2020learningbased}:
\begin{theo} \label{theo:existence}
    If $\rho_0 \in L^\infty{\Parens{\R, [0, 1]}}$ and $f \in C^2\left( [0, 1], \R_{\geq 0} \right)$, then there exist weak solutions $\rho$ of~\eqref{eq:hclaw} with regularity $\rho \in C^0{\Parens{\R_{\geq 0}, L^1_{\mathrm{loc}}{\Parens{\R, [0, 1]}} }}$.
\end{theo}
Uniqueness is ensured by restricting the set of weak solutions to those satisfying the Lax-E condition~\cite[Chap.~14]{Dafermos:1315649}.
Alternatively, one may consider the unique smooth solution of the related equation
\begin{equation} \label{eq:hclaw-visc}
    {\bar \rho}_t(t, x) + f{\Parens{\bar \rho(t, x)}}_x = \gamma^2 \bar \rho_{xx}(t, x)
\end{equation}
for small $\gamma > 0$.
If ${\bar\rho}_{\gamma}$ is the solution of~\eqref{eq:hclaw-visc}, then the limit ${\rho := \lim_{\gamma \downarrow 0} {\bar \rho}_\gamma}$ exists in $L^1_{\mathrm{loc}}$~\cite[p.~157]{Dafermos:1315649} and is the unique entropic solution of~\eqref{eq:hclaw}.

\subsubsection{Examples of flux functions}

The two most well-known flux functions used in traffic applications are perhaps the Greenshields~\cite{greenshields1935study} and the Newell-Daganzo~\cite{NEWELL1993281,DAGANZO1994269} flux functions.
The Greenshields flux function $f_{G}$ is given by 
\begin{equation} \label{eq:greenshields}
    f_G(\rho) = \rho v_G(\rho), \quad v_G(\rho) = V_f\Parens{1 - \rho}.
\end{equation}
At zero traffic density the vehicles move at the free flow velocity $V_f > 0$, and the velocity $v_G$ decreases linearly with density, going to zero as the road becomes more congested.

The Newell-Daganzo flux is function is given by 
\begin{equation} \label{eq:newell-daganzo}
    f_{\mathrm{ND}}(\rho) = \min{ \Set{V_f \rho, W{\Parens{1 - \rho}}} }.
\end{equation}
The velocity function in this case is then
\begin{equation*}
    v_{\mathrm{ND}}(\rho) = \begin{cases}
        V_f, & \rho < \sigma, \\
        W\frac{1 - \rho}{\rho}, & \rho \geq \sigma,
    \end{cases}
\end{equation*}
where $\sigma \in [0, 1]$ is such that $V_f \sigma = W{\Parens{1 - \sigma}}$.
Hence vehicles move at a constant average velocity up to a critical density $\sigma$ where the road becomes congested, and the velocity goes to zero as $\rho$ approaches one.

\begin{remark}
    The flux function in this case is not differentiable everywhere,
    but one can consider a smooth approximation of $f_{\mathrm{ND}}$ such that Theorem~\ref{theo:existence} applies.
\end{remark}

\subsection{Relation between the models}

The two models introduced so far are mathematically related.
The Greenshields model is a limit case of the first order follow-the-leader dynamics \cite{colombo2014} and the integral of the normalized density is proportional to the number of vehicles.
Consequently, $\rho$ is directly related to the intra-vehicular space and one can use the velocity function $v$ of the macroscopic scheme as the vehicle velocity $V$ \cite{liu2020learning}.
This leads to the cascaded system
\begin{equation} \label{eq:coupled}
    \left\{ \begin{array}{l}
        \rho_t(t, x) + f{\Parens{\rho(t, x)}}_x = \gamma^2\rho_{xx}(t, x), \\
        y_i'(t) = v{\Parens{ \rho(t, y_i(t))}}, \ i \in \{1, \dots, N\},
    \end{array} \right.
\end{equation}
for $t \geq 0$ and $x \in \R$, with appropriate initial conditions. In \cite{barreau2020learningbased} the authors give conditions for \eqref{eq:coupled} to be mathematically well-posed.

\begin{de} \label{def:consistency}
    For the two models to be self-consistent, one must ensure the following:
    \begin{enumerate}
        \item $f$ must be positive, concave and $C^2$;
        \item $v$ must be positive and decreasing;
        \item $v$ must be always larger than or equal to $f'$.
    \end{enumerate}
\end{de}

\begin{remark}
    The last item has the following physical motivation. If we consider that the flow is a wave, $f'$ is related to the group velocity (the speed of a characteristic) and $v$ to the phase velocity. Requiring $f' \leq v$ means that a particle must move faster than the envelop of the wave.
\end{remark}

The following proposition gives a simple sufficient condition for consistency of the models.
\begin{proposition} \label{prop:consistency}
    If $f$ is concave and $v \in C^2([0, 1], \R_{\geq 0})$, then equation~\eqref{eq:coupled} is a consistent model for describing micro- and macroscopic behaviors of traffic flow.
\end{proposition}

\begin{proof} 
    From the assumptions and using \eqref{eq:rhoV}, we get that $f(0) = 0$ and $f$ is positive. By concavity of $f$, the Chordal Slope Lemma implies that $v$ is decreasing. Since $f'(\rho) = v(\rho) + \rho v'(\rho)$, we get that $f' - v$ has the same sign as $v'$ which is negative. Consequently, Definition~\ref{def:consistency} indeed holds and the model is consistent.
\end{proof}

For the sequel, we need to express $f'$ and $f''$ as functions of $v$. Using \eqref{eq:rhoV}, we get: 
\begin{equation} \label{eq:fFv}
            f'(\rho) = v(\rho) + \rho v'(\rho), \quad
            f''(\rho) = 2 v'(\rho) + \rho v''(\rho).
\end{equation}
Combining \eqref{eq:coupled}, \eqref{eq:fFv}, Proposition~\ref{prop:consistency} and Definition~\ref{def:consistency} leads to the following for $i \in \{1, \dots, N\}$:
\begin{equation} \label{eq:coupled2}
    \hspace*{-0.3cm}
    \left\{ \begin{array}{l}
        \mathcal{N}_y[y_i] := y_i' - v \left( \rho(\cdot, y_i) \right) = 0, \\
        \mathcal{N}_{\rho}[\rho, v] := \rho_t + \left(v(\rho) + \rho v'(\rho) \right)\rho_x - \gamma^2 \rho_{xx} = 0, \\
        \mathcal{N}_v[v] := 2 v' + \rho v'' \leq 0, \\
        v \geq 0 \text{ and } C^2.
    \end{array} \right.
\end{equation}

\begin{remark}
    Note that with the Greenshields or the regularized Newell-Daganzo speed function, the two last statements of \eqref{eq:coupled2} are indeed verified, resulting in a consistent model.
\end{remark} 

\subsection{Measurements}

In this paper $N$ probe vehicles located at $y_1(t), \dots, y_N(t)$ are used as mobile sensors within the flow of vehicles, with dynamics given by the first equation in \eqref{eq:coupled}.
These vehicles can measure the following values in real-time:
\begin{enumerate}
    \item Their positions $y_i(t)$;
    \item The local density at their locations: $\rho_i(t) = \rho(t, y_i(t))$;
    \item And their instantaneous speed $v_i(t) = y_i'(t)$;
\end{enumerate}
Requiring position and speed measurements is reasonable since these can be obtained using GPS and/or inertial sensors.
The density measurements are more delicate.
In a single-lane environment, one can measure the inter-vehicular distances and deduce an approximation of the local density. 
In a multi-lane context, a camera or radar sensors might be needed~\cite{chellappa2004vehicle}. 

\begin{remark}
    Compared to the authors' previous work \cite{liu2020learning}, we have added the instantaneous speed measurements.
    These are however not essential but preferable since we do not make any assumptions on the velocity-density relation.
\end{remark}

Based on equation~\eqref{eq:coupled2}, one can investigate the state-reconstruction problem.

\section{Problem Statement}

The general definition of partial-state reconstruction is proposed in \cite{barreau2020learningbased}.
Here, we adapt this definition to the case of density reconstruction.

\begin{de} \label{def:stateReconstruction}
     Let $\Omega = [0, T]$, $T > 0$. A \textbf{density reconstruction} $\hat{\rho}$ in $\H_c \subseteq H^1(\Omega, H^2(\R, [0, 1]))$ is defined by
     \begin{equation} \label{eq:stateReconstruction}
         \hat{\rho} \ \in \ \Argmin_{\bar{\rho} \ \in \ \H_c} \sum_{i=1}^N \int_0^T  \left| \rho_i(t) - \bar{\rho}(t, y_i(t)) \right|^2 \, dt.
     \end{equation}
\end{de}

The overall quality of the reconstruction is measured by the generalization error (GE):
\begin{equation} \label{eq:GE}
    \text{GE}_T(\hat{\rho}) = \int_{0}^T \int_{y_1(t)}^{y_N(t)} \left| \rho(t, x) - \hat{\rho}(t, x) \right|^2 \, dx \, dt.
\end{equation}

\textbf{Problem statement:} The objective of this paper is to propose an efficient algorithm to compute a density reconstruction with low GE using noisy measurements from PVs.

To obtain a small GE, there are two options:
\begin{enumerate}
    \item There are many probe vehicles so that $\text{GE}_T(\hat{\rho}) \simeq \frac{1}{N} \sum_{i=1}^N \int_0^T  |\rho_i(t) - \hat{\rho}(t, y_i(t)) |^2 \, dt$ and a density reconstruction will lead to a small error;
    \item The set $\H_c$ is defined such that the minimizers have low GE.
\end{enumerate}

Of course, we assume that the penetration rate is low, meaning that there are few probe vehicles compared to the number of vehicles.
Hence the generalization error in the first case will be high if $\H_c = H^1(\Omega, H^2(\R, \R))$.

In this paper we are interested in the second case and the set $\H_c$ is constructed using the fact that \eqref{eq:coupled2} must hold. In other words, we define
\[
    \begin{array}{rl}
        \!\!\!\H_c =\!\!\!\!& \displaystyle \left\{ \bar{\rho} \in H^1(\Omega, H^2(\R, \R)) \ | \vphantom{\int_0^T}  \int_0^T \!\!\! \int_{y_1}^{y_N} \N_{\rho}[\bar{\rho},v]^2 = 0 \text{ for } \right.\\
        & \displaystyle \quad \quad v \in \Argmin_{\bar{v} \in C^2([0,1], \R_{\geq 0})} \sum_{i=1}^N \int_0^T |v_i(t) - \bar{v}(\rho_i(t))|^2 dt \\
        & \left.\displaystyle \quad \quad \quad \quad \text{ s. t. } \int_0^1 \max \left( 0, \N_v[\bar{v}]\Big\rvert_{\rho = u} \right)^2 \, du = 0 \right\}.
    \end{array}
\]

\begin{remark} \label{rem:observability}
    It has been proven in~\cite{delle2019traffic} (with $\gamma = 0$) that the conditions of Definition~\eqref{def:consistency} imply perfect reconstruction for finite $T > 0$. For any density reconstruction $\hat{\rho} \in \H_c$, we get $\text{GE}_t(\hat{\rho}) = \text{GE}_T(\hat{\rho})$ for all $t \geq T$.
\end{remark}

In the following section we derive a relaxed version of~\eqref{eq:stateReconstruction} and discuss how to locally solve it.

\section{Learning-based Density Reconstruction}
\label{sec:learning}

There are several ways of solving constrained optimization problems.
One is to define the optimized variables in such a way that the constraints are already satisfied as done in \cite{lagaris1998artificial} where the boundary conditions are enforced in the solution.
However, this technique does not work well in the presence of noise or for complex systems since the boundary condition cannot be enforced.
We will explore here another technique which consists of penalizing unfeasible solutions using an extended Lagrangian cost function.

\begin{remark}
    We assume here that the measurements over the interval~$[0, T]$ are given as a finite dataset consisting of $\Nmeas$ values at the sampling instants $\Set{t_k}_{k=1}^{\Nmeas}$.
\end{remark}

The Lagrange multiplier method relaxes the optimization problem \eqref{eq:stateReconstruction} to:
\begin{equation} \label{eq:langrangeMultiplier}
    \hat{\rho} \in \Argmin_{\bar{\rho} \ \in \ H^1(\Omega, H^2(\R, [0, 1])} \left\{ \min_{v \in C^2 \left( [0,1], \R_{\geq 0} \right), \ \gamma > 0 } \ \sum_k \lambda_k \L_k \right\}
\end{equation}
for $\lambda_k \geq 0$ and where the costs $\L_k$ are given in the sequel.

\subsection{Interpretation of the different costs}

\subsubsection{Data-based costs}\label{sec:databased}

The first kind of cost function is typical in regression problems in machine learning~\cite{Goodfellow-et-al-2016}.
It consists of the mean square error (MSE) between the measurements and the estimated function.
Given the three types of measurements described above, we arrive at the following cost functions:
\begin{enumerate}
    \item We measure the position of each PV at given instants $t_k$ so that the dataset is $\{ y_i(t_k) \}_{i,k}$ and the MSE is
    \[
        \L_1 = \frac{1}{N} \sum_{i=1}^N \frac{1}{\Nmeas} \sum_{k=1}^{\Nmeas} \Bracket{ y_i(t_k) - \hat{y}_i(t_k) }^2.
    \]
    Here, $\hat{y}_i : [0, T] \to \R$ is the estimated position of the $i$-th PV. 
    
    \item The second data-based cost function is related to the density measurements $\{ \rho_i(t_k) \}_{i,k}$:
    \[
        \hspace*{-0.5cm}\L_2 = \frac{1}{N} \sum_{i=1}^N \frac{1}{\Nmeas} \sum_{k=1}^{\Nmeas} \Bracket{ \Parens{\rho_i(t_k) - n_{\rho_i}} - \hat{\rho}(t_k, y_i(t_k)) }^2\!\!\!,
    \]
    where $\hat{\rho} : [0, T] \times \R \to [0, 1]$ is the estimated density and $n_{\rho_i}$ is a variable introduced to suppress the bias in the noise (as discussed in \cite{barreau2020learningbased}).
    
    \item Finally, including the velocity measurements enables the identification of the model~\eqref{eq:coupled2} by estimating the velocity function:
    \[
        \L_3 = \frac{1}{N} \sum_{i=1}^N \frac{1}{\Nmeas} \sum_{k=1}^{\Nmeas} \Bracket{ v_i(t_k) - \hat{v}( \rho_i(t_k) ) }^2,
    \]
    where $\hat{v} : [0, 1] \to \R_{\geq 0}$ is the estimated velocity.
\end{enumerate}
For $n_{\rho_i} = 0$, these cost functions are based on the measurement data only, so they provide quantitative information which can be used to reject \emph{unbiased} noise. 

\begin{remark}
    There is no need to remove the biases on the trajectories since they are defined up to a constant.
    There is also no correction on the velocity measurements since they are assumed to be noiseless.
\end{remark}

However, this type of cost function has some drawbacks.
In the presence of biased noise, the previous cost functions do not perform well.
Secondly, we must have as many density measurements as position and speed measurements.
One solution to these problems is to introduce two other cost functions:
\begin{enumerate}
    \item The first corrects for errors in the trajectory measurements:
    \[
        \hspace*{-0.75cm}
        \L_4 = \frac{1}{N} \sum_{i=1}^N \frac{1}{\Nmeas} \sum_{k=1}^{\Nmeas} \Bracket{ \Parens{\rho_i(t_k) - n_{\rho_i}} - \hat{\rho}(t_k, \hat{y}_i(t_k)) }^2.
    \]
    \item The second cost function is related to velocity measurements:
    \[
        \hspace*{-0.75cm}
        \L_5 =\! \frac{1}{N} \sum_{i=1}^N \frac{1}{\Nmeas} \! \sum_{k=1}^{\Nmeas} \Bracket{ v_i(t_k) - \hat{v}{\Parens{ \hat{\rho}_i(t_k, \hat{y}_i(t_k)) }}  }^2\!\!\!.
    \]
\end{enumerate}

Finally, the most important deficiency when using only data-based costs is the necessity of a large dataset. More specifically, the measurements should be well spread across the set $[0, T] \times [y_1, y_N]$.
This is not the case in our problem since the measurements are taken along the trajectories of the PVs.
To ensure a small generalization error we need to consider additional cost functions.


\subsubsection{Physics costs}

These cost functions do not depend on the measurements, but only on the reconstructed functions.
Consequently, they can be seen as regularizers.

\begin{enumerate}
    \item Dynamics of the traffic density: given $\Nphy^\rho$ sample points $\Set{(t^{\rho}_k, x^{\rho}_k)}_{k = 1}^{\Nphy^\rho}$,
        \begin{equation*}
            \L_6 = \frac{1}{\Nphy^{\rho}}\sum_{k = 1}^{\Nphy^\rho}{
            \Parens{
                \N_\rho{\Bracket{\hat{\rho}, \hat{v}}}\Big\rvert_{(t, x) = (t^{\rho}_k, x^{\rho}_k)}
            }^2
            }.
        \end{equation*}
        In addition, the dissipation coefficient $\gamma$ that is introduced to smooth $\hat{\rho}$ should be as small as possible,
        leading to the introduction of another cost:
        \begin{equation*}
            \L_7 = \gamma^2.
        \end{equation*}
    \item Dynamics of the PV trajectories: given $\Nphy^y$ sample points $\Set{t^{y}_k}_{k = 1}^{\Nphy^y}$,
        \begin{equation*}
            \L_8 = \frac{1}{N}\sum_{i = 1}^N \frac{1}{\Nphy^y} \sum_{k = 1}^{\Nphy^y}{
            \Parens{
                \N_y{\Bracket{\hat{y}_i}}\Big\rvert_{t = t^{y}_k}
            }^2
            }.
        \end{equation*}
    \item Constraint on the concavity of $f$: given $\Nphy^v$ sample points $\Set{\rho^{v}_k}_{k = 1}^{\Nphy^v}$,
        \begin{equation*}
            \L_{9} = \frac{1}{\Nphy^v}\sum_{i=1}^{\Nphy^v}{
                \max{\Parens{0,
                    \N_v{\Bracket{\hat{v}}}\Big\rvert_{\rho = \rho_k^v}
                }}^2
            }.
        \end{equation*}
\end{enumerate}

These three costs are based on the assumption that the real dynamics follow equation~\eqref{eq:coupled2}.
The advantages are the same as in any model-based framework: they can effectively reduce the noise.

\subsubsection{Total cost}

The total cost is a weighted sum of the ten previously introduced loss functions weighted by the $\lambda_k$.
If $\lambda_2, \lambda_3, \lambda_6$ and $\lambda_8$ are all strictly positive, the penalty method \cite{bertsekas2014constrained} states that solving a sequence of problems with these weight variables increasing to infinity leads to a solution of the discretization of problem~\eqref{eq:stateReconstruction}.

The remaining cost functions are used to enhance the quality of the reconstruction by restricting the set of possible solutions and rejecting the noise.
Indeed, the trajectory reconstructions suppress unbiased noise on the locations of the local density measurements.
The costs related to the velocity are used for identification.
It appears that using a subset of the velocity measurements is enough since the identification is also supported by $\L_5, \L_6, \L_8$ and $\L_{9}$. 

\begin{remark}
    The physics-based cost functions we considered above are based only on the dynamics~\eqref{eq:coupled2} and do not depend on the measurement data at all.
    We could additionally consider `mixed' data- and model-based cost functions.
    For instance, relating the density measurements to the vehicle dynamics one can consider the cost function
    \begin{equation*}
       \L_{\text{mixed}} = \frac{1}{N} \sum_{i=1}^N \frac{1}{\Nmeas} \sum_{k=1}^{\Nmeas} \Bracket{
           \hat{y}_i'(t_k) - \hat{v}{\left( \rho_i(t_k) - n_{\rho_i} \right)}
       }^2,
    \end{equation*}
    allowing for simultaneous noise rejection and identification.
    This type of cost is referred to as a \emph{first-order} physics cost, while those considered above are called \emph{second-order} physics costs.
    Note that unlike the second-order physics costs, this type of cost cannot correct for the problem of sparse measurements, since it depends on the data.
    
    We observed that this type of cost is not necessary to obtain good performance in the numerical examples we consider here, hence in this paper the total cost function $\L$ contains only second-order physics costs.
\end{remark}


\subsection{Neural network approximation of a density reconstruction}

Let $\Theta_{\theta}$ be a general neural network, where $\theta$ is a tensor containing the parameters of the network.
We propose to approximate a function $\hat{\rho}$ satisfying \eqref{eq:langrangeMultiplier} by $\Theta_\theta$.
It has been proved in \cite{sirignano2018dgm} and \cite[extended version]{barreau2020learningbased}
that $\hat{\rho}$ can be approximated arbitrarily well by $\Theta_{\theta}$ provided that the number of neurons is large enough.
The main issue becomes solving \eqref{eq:langrangeMultiplier} for $\hat{\rho} = \Theta_{\theta}$.

\subsubsection{Training procedure}

The optimization problem \eqref{eq:langrangeMultiplier} is particularly difficult to solve since it involves many `adversarial' costs. In that scenario, standard optimization algorithms might not minimize the costs that are slowly varying with respect to $\theta$.
To the best of the authors' knowledge, the article~\cite{wang2020understanding} is the first to investigate this problem by adapting the weights $\{ \lambda_k \}$ during the training.
However, the weight update algorithm is sensitive to noise and based on empirical considerations.
Based on the same idea, in \cite{barreauGeneral}, the authors propose a set of algorithms to deal with this issue
based on algorithms for constrained optimization.

\begin{figure}
    \centering
    \includegraphics[width=0.45\textwidth]{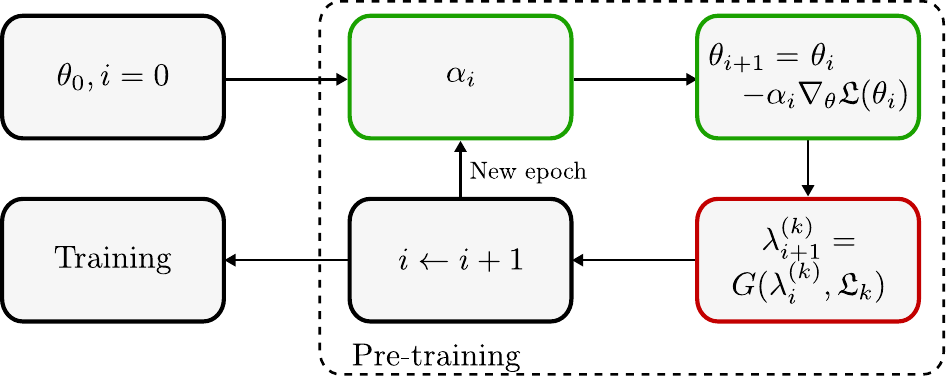}
    \caption{Flow diagram of the update process. The boxes in green refer to traditional 1st order optimization scheme. The red box refers to the the hyper-parameters $\lambda_k$ update. The function $G$ is defined in \cite{barreauGeneral}.}
    \label{fig:update}
    \vspace*{-0.50cm}
\end{figure}

The overall training procedure is summarized in Figure~\ref{fig:update} and is divided into two parts:
\begin{itemize}
    \item Pre-training step: we use a first order algorithm (such as stochastic gradient descent or ADAM) combined with a suitable algorithm for updating the $\lambda_i$ for a fixed number of epochs.
    \item Training step: 
    with the $\lambda_i$ now fixed at appropriate values, a second-order optimizer (e.g., BFGS) is used to ensure convergence to a global optimum.
\end{itemize}
More details about the weight update algorithm can be found in Appendix~\ref{appendix:training}.

The main disadvantage of the proposed training procedure is that there is no guarantee on the convergence to a global minimum of \eqref{eq:langrangeMultiplier}, as well as no upper bound on the GE.
Only convergence to a local minimum in parameter space is theoretically guaranteed
(as is typically the case when using deep neural networks, since the resulting optimization problem is usually non-convex).

\subsubsection{Network architecture}

As discussed in~\cite{barreau2020learningbased}, fully connected deep neural networks with hyperbolic tangent activation functions provide an effective and simple architecture for this problem.
%

The network used to approximate the density has 5 hidden layers with 10 neurons each.
In order to ensure that the network output remains bounded, a $\tanh$ function is applied after the final layer.
%
For the trajectory reconstruction, we use a small neural network with 3 hidden layers and 5 neurons per layer for each probe vehicle.

The speed function $v$ is also approximated by a neural network.
Since this function should not be too complex (otherwise it will be very difficult to solve the PDE),
we use a very simple multi-layer neural network with 2 hidden layers and 5 neurons per layer.
We enforce ${\hat{v}(1) = 0}$ by multiplying the network output by ${1 - \rho}$.

\section{Simulation Results and Discussion}

In this section, we first compare simulation results and then investigate the performance of the training procedure using a statistical analysis\footnote{The code and data are available at \url{https://github.com/mBarreau/TrafficReconstructionIdentification}}.

\subsection{SUMO simulation}

Figure~\ref{fig:rho} shows the traffic density obtained from a SUMO simulation (see Appendix~\ref{appendix:SUMO} for more details).
One can clearly see stop and go waves which originate from a traffic light located at $x = 2.5$~km.

The neural network reconstruction gives the estimated density in Figure~\ref{fig:rhoHat}.
A low number of physics points is used ($\Nphy^\rho = 500$) as SUMO uses second order follow-the-leader dynamics and approximating it using a first order hyperbolic equation leads to discrepancies.
With low values of $\Nphy^{\rho}$, the model is enforced only when data is not available, thus ensuring a relatively good fit \cite{barreauGeneral}.
The characteristics are well-identified and the generalization error is small ($\approx 0.3$),
although the result is smoother than the original density.

\begin{figure}[htbp]
    \centering
    \subfloat[Normalized density from a SUMO simulation. Black lines represent the PV trajectories.]{\includegraphics[width=.45\textwidth]{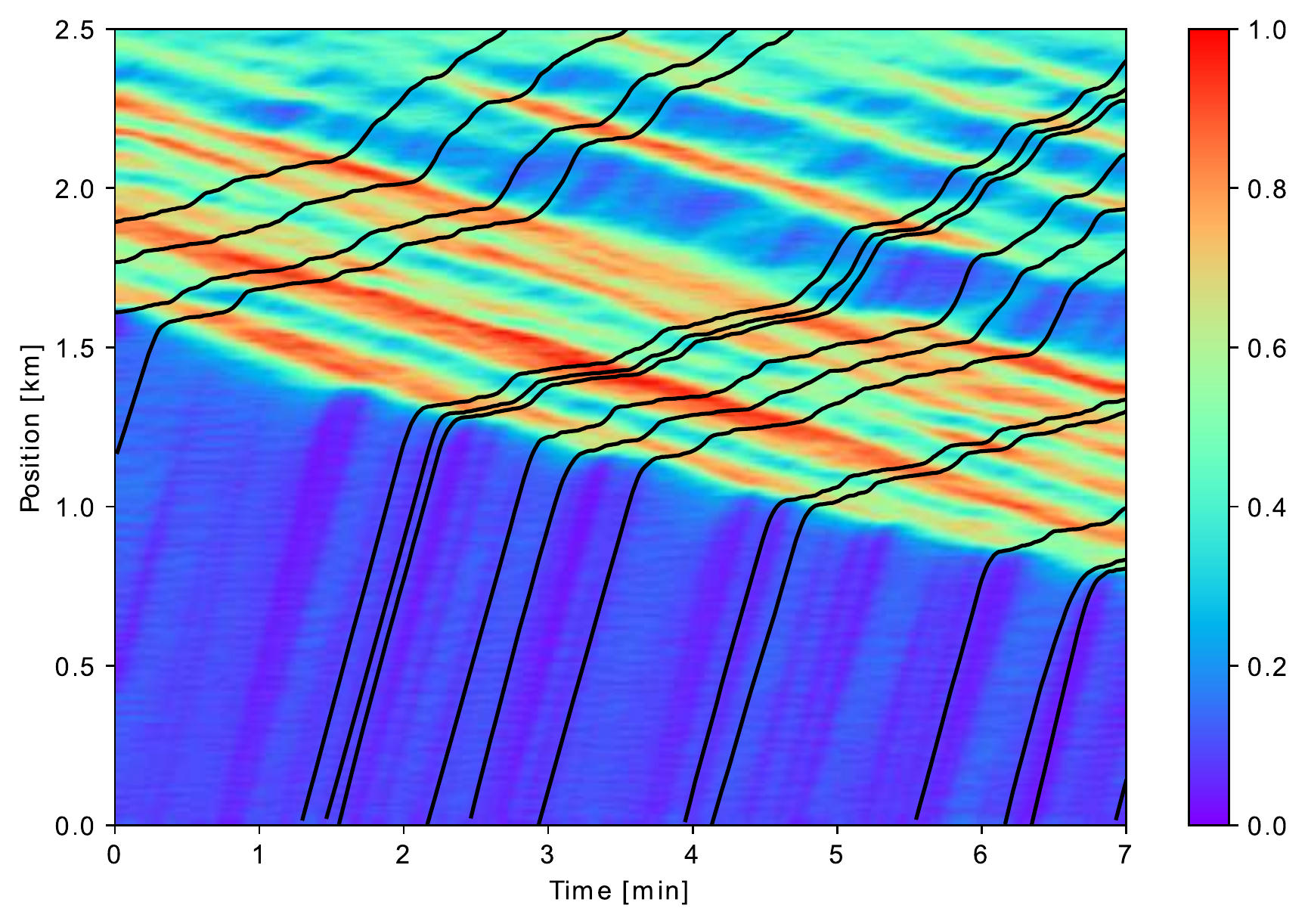}\label{fig:rho}} 
    \vspace*{-0.3cm}\\
	\subfloat[Reconstructed density using a neural network. Brown lines represent the estimated PV trajectories.]{
	\includegraphics[width=.45\textwidth]{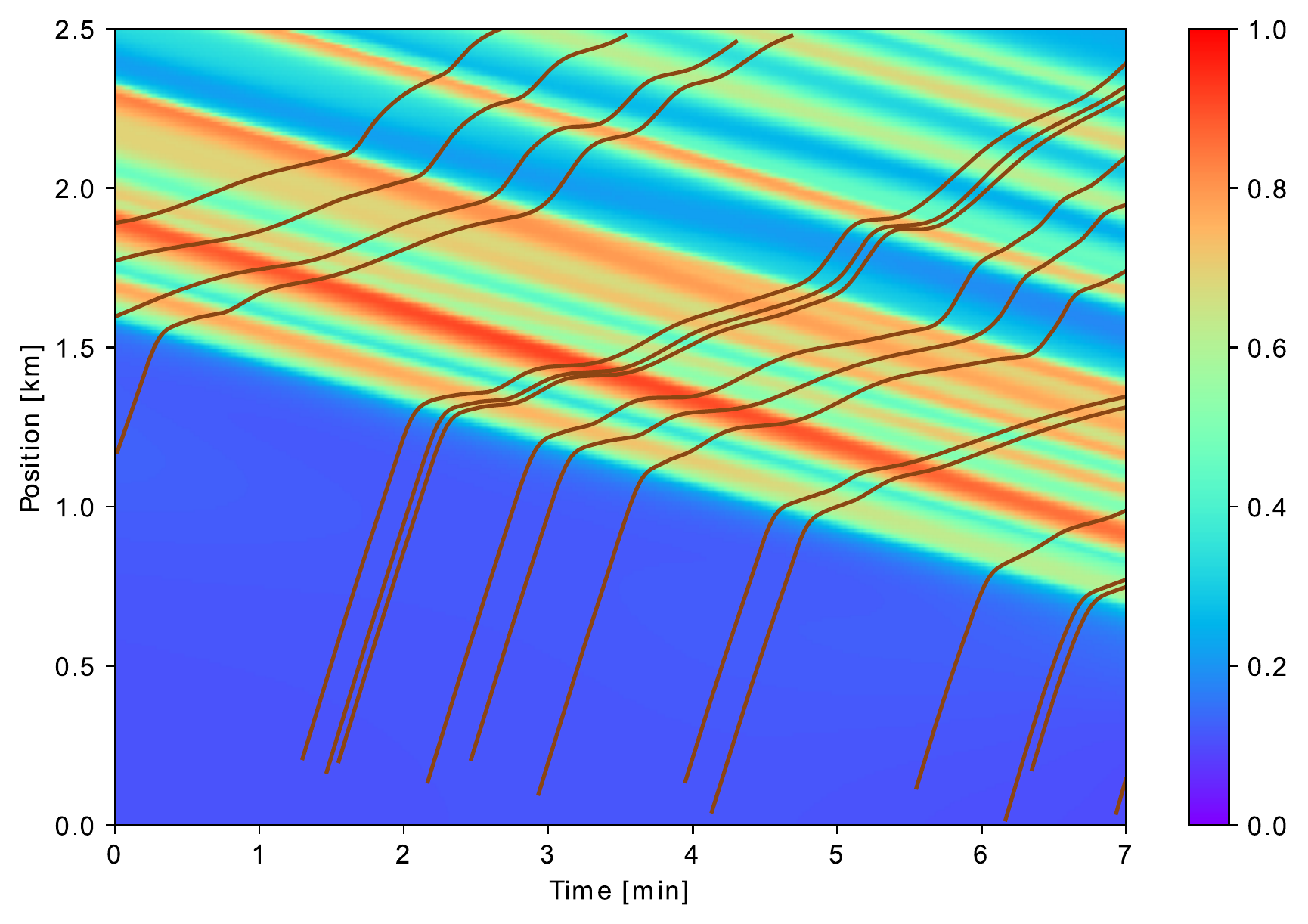}\label{fig:rhoHat}}
	\caption{Density reconstruction using data from a SUMO simulation. }
    \vspace*{-0.50cm}
\end{figure}

\subsection{Training performance}\label{subsec:training}


We investigate here the computation time and error for three simulation cases:
\begin{enumerate}
    \item density measurements computed using a Godunov solver of~\eqref{eq:hclaw} with the Greenshields flux function~\eqref{eq:greenshields};
    \item density measurements computed using a Godunov solver of~\eqref{eq:hclaw} with the Newell-Daganzo flux function~\eqref{eq:newell-daganzo};
    \item density measurements from the SUMO simulation described above.
\end{enumerate}
We compute the density reconstruction with and without the pre-training step described in~\ref{subsec:training}.
The distribution of the computation time and the generalization error~\eqref{eq:GE} for 50~simulation runs of each case is shown in Figure~\ref{fig:simulation_stats}.
The simulations were run on a laptop with an Intel(R) Core(TM) i5-8365U CPU~@~1.60~GHz with four processor cores.

In the first two cases, we observe that using the pre-training leads to slightly higher computational time but the mean and variance of the generalization error decrease significantly.
This implies that the pre-training step improves the solution quality and makes the training procedure more robust with respect to the parameter initialization.

In the third case, the performance is worsened by the pre-training step, due to the fact that this step attempts to enforce the dynamics~\eqref{eq:coupled2} more strictly.
As mentioned above, the first-order dynamics we considered here do not match those used by the SUMO simulator, so we do not observe the same performance improvement as in the previous two cases where the model and data where generated by the same dynamics.

\begin{figure}[htbp]
    \centering
    \includegraphics[width=0.49\textwidth]{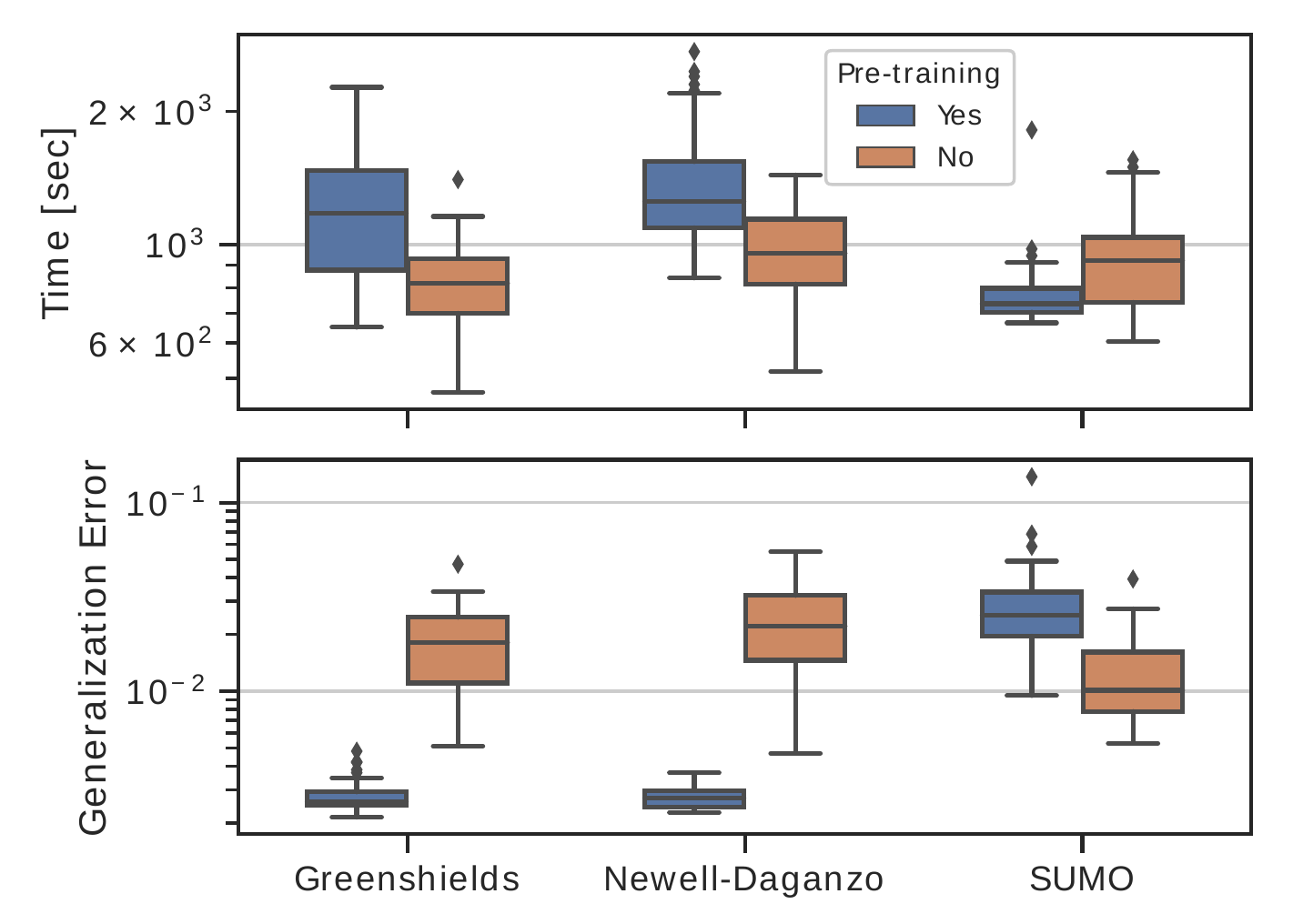}
    \caption{Distribution of the computation time and GE for the three simulation cases described in~\ref{subsec:training}}
    \label{fig:simulation_stats}
    \vspace*{-0.20cm}
\end{figure}


\section{Conclusion \& Perspectives}

In this article, we have investigated the use of physics-informed learning for identification and state reconstruction of traffic flow.
The results show a good reconstruction quality with a moderate computational burden.
The proposed training procedure helps in dealing with many physics costs and ensures a faster and more robust convergence.
However, there are some discrepancies in the reconstruction using the data from SUMO, which are due to the model mismatch.
A second order scheme with speed dynamics should be investigated.
The scalability of the method to a large number of PVs is also to be studied, as the method requires one small neural network per PV.
In addition, the extension to predictions in the near future should be considered,
and comparisons with other methods~\cite{seo2017traffic,delle2019traffic,cicic2020numerical} have to be conducted.

\appendix

\subsection{The training procedure}
\label{appendix:training}

A complete study of training procedures useful in physics-informed learning is done in \cite{barreauGeneral}.
Here we give a brief summary of the techniques used in this article.

With a gradient-descent based optimization algorithm, the weights are updated according to
\[
    \theta_{i+1} = \theta_i - \alpha_i \sum_{k=1}^{9} \lambda_i^{(k)} \nabla_{\theta} \L_k(\theta_i),
\]
where $\alpha_i$ is the learning rate.

The main idea (developed initially in \cite{fioretto2020lagrangian,nandwani2019primal}) is to perform a dual update on the $\lambda_i$ during training:
\[
    \lambda_{i+1}^{(k)} = \lambda_i^{(k)} - \alpha_{\lambda} \nabla_{\lambda^{(k)}} \L_k(\theta_i).
\]

The update of a weight $\lambda$ depends on the nature of the associated cost function.
Data-based costs (as in part~\ref{sec:databased}) have the same importance throughout the entire training and can provide a good estimate in the initial training steps, so the corresponding weight should be positive and constant.
Physics-based costs, meanwhile, represent constraints which can be of two types:
\begin{enumerate}
    \item a soft constraint: the cost cannot be reduced to zero (finite-size neural network, noisy measurements...).
    Consequently, this term must be weighted such that its contribution relative to the data-based costs at the optimum is given by the user.
    This prior knowledge reflects the feasibility and reliability of the constraint.

    \item a hard constraint: the cost must be reduced to (approximately) zero.
    These constraints are typically related to some physical law underlying the data that must hold.
    After training, the cost value must be as small as possible, regardless of the values of the other costs.
\end{enumerate}
These two kinds of costs require different weight update procedures.
For a soft constraint, the weight is upper bounded and should approach the desired final value given by the user.
On the contrary, for a hard constraint, the weight is not upper bounded and the growth rate depends on the value of the corresponding cost.

\subsection{The SUMO simulation}
\label{appendix:SUMO}

The SUMO simulation is relatively simple and the full code is available on \texttt{\href{https://github.com/mBarreau/TrafficReconstructionIdentification}{GitHub}}. The road is of length $L = 3$~km with a traffic light at $x = 2.5$~km.
Since the traffic light is not represented in the model, we propose here to reconstruct the traffic for $x \in [0, 2.5]$~km.
The car dynamics are described in \cite{behrisch2011}.
The probability of a car being a probe vehicle is set to $0.1$.

The output of the simulation is the trajectories of all simulated vehicles.
From this dataset, we reconstruct the number of cars in a given cell on a grid.
The density is then obtained through convolution with a Gaussian filter~\cite{scheepens2011composite}.
The parameters of the Gaussian kernel depend on the grid and the speed of a vehicle in a given context.
Here we assume the parameters to be the same, independently of the traffic state: $\sigma_{\text{space}} = 0.01$~km and $\sigma_{\text{time}} = 0.06$~min.

\bibliographystyle{IEEEtran}
\bibliography{biblio}

\end{document}